\documentclass[12pt]{amsart}
\usepackage[margin=1.2in]{geometry}

\title{Stability of Sampling for CUR Decompositions}

\usepackage{hyperref}
\usepackage{tablefootnote}
\usepackage{graphicx}
\usepackage{subcaption}
\usepackage{amsthm}
\usepackage{dsfont}
\usepackage{amssymb}
\usepackage{enumerate}
\usepackage{graphicx}
\usepackage{mathrsfs}
\usepackage{float}
\usepackage{bbm}
\usepackage{amsmath}
\usepackage{comment}
\usepackage{hyperref}
\usepackage{listings}
\usepackage{color}
\usepackage{ulem}
\usepackage[dvipsnames]{xcolor}
\usepackage[lined,boxed,commentsnumbered]{algorithm2e}
\usepackage{fourier}
\usepackage{multicol}
\usepackage{soul}
\allowdisplaybreaks

\hypersetup{
	colorlinks,
	citecolor=blue,
	filecolor=blue,
	linkcolor=blue,
	urlcolor=blue,
	hyperfootnotes=false
}

\newtheorem{theorem}{Theorem}[section]
\newtheorem{proposition}[theorem]{Proposition}
\newtheorem{lemma}[theorem]{Lemma}

\newtheorem{corollary}[theorem]{Corollary}

\newtheorem{thmnum}{Theorem}

\newtheorem{cornum}{Corollary}

\theoremstyle{definition}

\newtheorem{experiment}{Experiment}
\newtheorem{remark}[theorem]{Remark}

\newcommand{\R}{\mathbb{R}}

\newcommand{\N}{\mathbb{N}}

\newcommand{\K}{\mathbb{K}}

\newcommand{\E}{\mathbb{E}}
\newcommand{\Prob}{\mathbb{P}}

\newcommand{\rank}{{\rm rank\,}}

\newcommand{\eps}{\varepsilon}

\begin{document}

\author{Keaton Hamm}
\address{Department of Mathematics, University of Arizona, Tucson, AZ 85719 USA}
\email{hamm@math.arizona.edu}

\author{Longxiu Huang}
\address{Department of Mathematics, University of California, Los Angeles, CA 90095 USA}
\email{huangl3@math.ucla.edu}

\keywords{CUR Decomposition, Low Rank Matrix Approximation, Dimensionality Reduction, Column Subset Selection, Randomized Sampling}
\subjclass[2010]{15A23,65F30,68W20}


\begin{abstract}


This article studies how to form CUR decompositions of low-rank matrices via primarily random sampling, though deterministic methods due to previous works are illustrated as well.  The primary problem is to determine when a column submatrix of a rank $k$ matrix also has rank $k$.  For random column sampling schemes, there is typically a tradeoff between the number of columns needed to be chosen and the complexity of determining the sampling probabilities.  We discuss several sampling methods and their complexities as well as stability of the method under perturbations of both the probabilities and the underlying matrix.  As an application, we give a high probability guarantee of the exact solution of the Subspace Clustering Problem via CUR decompositions when columns are sampled according to their Euclidean lengths.

\end{abstract}

\maketitle


\section{Introduction}


Low-rank matrices have taken on an important role in recent years both in theory and in applications as it has been observed that data matrices arising in diverse applications are very well-approximated by low-rank matrices \cite{UdellTownsend2019LowRank}.  Furthermore, matrix factorization methods based on low-rank structure have been used to great effect to solve linear systems, ce{}ompress data, speed up computations, and elucidate structure of matrices.  In fact, matrix factorizations appear twice in the list of top ten algorithms of the twentieth century \cite{dongarra2000guest}.  As theory and application of Machine Learning advances at an exponential rate, explorations of low-rank methods do as well on account of their success and fundamental importance.  

When designing a low-rank approximation method for practice, there are several factors that one may consider, including storage cost, computational complexity, and interpretability.  The latter is the subject of \cite{DMPNAS}, in which the authors propose the use of the CUR decomposition for matrices as a way to perform dimensionality reduction on a given set of data while maintaining interpretability of the results.  That is, using the Singular Value Decomposition (SVD) as is done in Principal Component Analysis (PCA) can lead to a representation of the data in terms of an abstract basis, and hence the resulting representation may lose interpretability (e.g., what is an eigenpatient in a medical trial).  These methods are useful in many tasks including prominent clustering algorithms like Spectral Clustering, but are not always suitable for this reason.  An alternative is to try to use the \textit{self-expressive} property exhibited by many datasets and attempt to use actual columns of the data as a dictionary in which to represent it.  This task, called \textit{column selection}, may be thought of as in-data feature selection which attempts to find the most representative data points to capture the salient features of the data.  Of course one may perform row selection as well, in which case an approximation of the form $A\approx CUR$ can be made in which the matrices $C$ and $R$ are column and row submatrices of $A$ itself; approximations of this form are called CUR approximations ($U$ is chosen in various ways which will be discussed in the sequel).

There are both deterministic and random methods for forming CUR approximations, each of which have advantages and drawbacks.  For large matrices, random sampling is typically less computationally expensive, but requires that more columns and rows be selected to guarantee good performance.  On the other hand, better theoretical guarantees may be given for certain deterministic column selection procedures.  Many works on column and row selection use CUR approximations as a fast way to approximate the truncated SVD, e.g., \cite{BoutsidisOptimalCUR,DKMIII,DM05, DMM08}.  However, some of the proposed algorithms in the literature perform well asymptotically, but do not guarantee recovery of actual low-rank matrices (for instance, that of \cite{DKMIII}).  In this article, we consider two main items: first, when do deterministic or random sampling procedures give rise to exact CUR decompositions for low-rank matrices (i.e., $A=CUR$), and second, are such methods stable under perturbations of either the underlying matrix or the sampling probabilities in the random case. As a sample application of our analysis, we illustrate how one can use our sampling guarantees to solve the Subspace Clustering Problem via some known matrix factorization methods.

\section{Main Results}

We consider the problem of how to select columns and rows to obtain an exact CUR decomposition of a low rank matrix, and prove that this procedure is stable under small perturbations.  Note that if $C$ and $R$ are column and row submatrices of $A$ which has low-rank, and $U$ is the matrix formed from entries where $C$ and $R$ overlap -- i.e., if $C=A(:,J)$ and $R=A(I,:)$ then $U:=A(I,J)$ -- then the classical statement of the CUR decomposition is that $A=CU^\dagger R$ if and only if $\rank(U)=\rank(A)$.  This exact decomposition goes back at least as far as the 1950s \cite{Penrose56} in the case that $U$ is square and invertible (this case also follows from rank additivity for Schur decompositions \cite{guttman1946enlargement}); for a history, the reader is invited to consult \cite{HammHuangACHA}, but the main theorem therein which characterizes this exact decomposition is restated in Section \ref{SEC:Background}. Our initial sampling result is obtained from some established results of Rudelson and Vershynin \cite{Rudelson_2007}.  Here we state simplified versions of the results to indicate their flavor to the reader, and reference the full statement that appears later.
\begin{thmnum}[Theorem \ref{THM:ColRowChnM}]\label{THM:IntroB}
If $A$ has rank $k$, then sampling $O(k\log k)$ columns and rows of $A$ independently with replacement according to column and row lengths, respectively, implies that $A = CU^\dagger R$ with high probability.
\end{thmnum}
Our method of proving Theorem \ref{THM:IntroB} allows for low sampling complexity (note that at least $k$ rows and columns must be sampled to achieve a valid CUR decomposition, so we incur only an extra $\log k$ factor) and also allows columns and rows to be sampled independently of each other.  Moreover, our proof technique allows us to demonstrate stability of this sampling method in the following sense.

\begin{thmnum}[Theorem \ref{THM:StableCUR}]\label{THM:IntroC}
If $A$ has rank $k$, and $p_i,q_i$ are probability distributions determined by the column and row lengths of $A$, respectively, then for any probability distributions which satisfy $\widetilde{p}_i\geq\alpha_ip_i$, $\widetilde{q}_i\geq\beta_iq_i$ for some $\alpha_i,\beta_i>0$, sampling $O(k\log k)$ columns and rows of $A$ independently with replacement according to $\widetilde{p}$ and $\widetilde{q}$, respectively, implies that $A=CU^\dagger R$ with high probability.
\end{thmnum}

As a corollary, we find that uniform sampling of rows and columns yields an exact CUR decomposition with high probability; this result is new: the only previous results for uniform sampling were given by Chiu and Demanet under coherence assumptions on the columns of $A$ \cite{DemanetWu}.  
Additionally, we may combine Theorems \ref{THM:IntroB} and \ref{THM:IntroC}: suppose that $\widetilde{A} = A+E$ where $A$ has rank $k$, and we sample columns and rows of $\widetilde{A}$ to form $\widetilde{C},\widetilde{R},$ and $\widetilde{U}$.  These may be written as $\widetilde{C} = C+E(:,J)$, for instance, where $C,R,$ and $U$ are the corresponding column, row, and intersection submatrices of the low rank matrix $A$.  It is natural to ask what the likelihood is that sampling from the noisy version of $A$ yields a CUR decomposition of $A$ itself. 

\begin{cornum}[Corollary \ref{COR:Uniform}]\label{COR:IntroA}
Suppose that $\widetilde{A}=A+E$, with $A$ having rank $k$.  Suppose also that no column or row of $\widetilde{A}$ is zero when the corresponding column or row of $A$ is nonzero.  Then sampling $O(k\log k)$ columns and rows of $\widetilde{A}$ uniformly with replacement yields $\widetilde{C},\widetilde{U},\widetilde{R}$ such that $A=CU^\dagger R$ with high probability.
\end{cornum}

Something more general than Corollary \ref{COR:IntroA} may be said: indeed if the noise is small compared to the matrix $A$, then sampling $\widetilde{A}$ according to its row and column lengths yields the same conclusion that $A=CU^\dagger R$ above.  This allows one to conclude that the error of $\widetilde{A}-\widetilde{C}\widetilde{U}^\dagger\widetilde{R}$ is on the order of $\|E\|$ via the perturbation results of \cite{HammHuangPer}.  

We also consider stability of sampling in terms of leverage scores instead of column and row lengths as in Theorem \ref{THM:IntroC}.  It was already known that sampling $O(k\log k)$ columns and rows via Leverage Score probabilities yields a valid CUR decomposition with high probability \cite{DMPNAS}, but we show that this is also stable as follows.

\begin{thmnum}[Corollary \ref{COR:LeverageStability}]\label{THM:IntroLev}
If $A$ has rank $k$, and $p_i, q_i$ are the Leverage Score probability distributions over the columns and rows of $A$, respectively, then for any probability distributions which satisfy $\widetilde{p}_i\geq\alpha p_i$, $\widetilde{q}_i\geq\alpha q_i$ for some $\alpha>0$, sampling $O(k\log k)$
columns and rows independently with replacement according to probabilities $\widetilde{p}$ and $\widetilde{q}$ implies that $A=CU^\dagger R$ with high probability.
\end{thmnum}

With stability results according to the sampling schemes, we can prove a guarantee for solving the Subspace Clustering problem via randomized sampling; see Corollary \ref{COR:SC}.

\subsection{Prior Works}

The analogue of Theorem \ref{THM:IntroB} was proven for Leverage Score sampling \cite{DMPNAS} (see Section \ref{SEC:RandomSampling} for the definition), but these are expensive to compute exactly as it requires computing the truncated SVD of $A$, and moreover the sampling complexity there is higher order in $k$.  Uniform sampling guarantees under incoherence assumptions on the columns of a matrix were given by Chiu and Demanet \cite{DemanetWu} for CUR approximations.  Most of the CUR approximation literature considers the case when $A$ has full rank and one forms an approximation $A\approx CUR$ where $U$ can take a variety of forms.  Typical estimates are in terms of the truncated SVD of $A$, i.e., of the additive error form: sampling $O(f(k,\eps))$ columns and rows to yield $\|A-CUR\|\leq \sigma_{k+1}(A) + O(\eps)\|A\|$, or relative error form: sampling $O(g(k,\eps))$ columns and rows to yield $\|A-CUR\|\leq (1+O(\eps))\sigma_{k+1}(A)$, where typically $f$ and $g$ are polynomial in $k$ and $\eps^{-1}$ and logarithms of these factors.  An incomplete but representative list of papers in this vein are \cite{DKMIII,DMM08,DMPNAS}.  Finally, the case when $A$ is square and symmetric positive semidefinite falls under the purview of the Nystr\"{o}m method, and finds abundant applications in Machine Learning due to the use of kernel matrices there.  Some references in this line are \cite{DM05,gittens2011spectral,GittensMahoney,BeckerNystrom}, but we note that there are significant differences in CUR approximations compared to Nystr\"{o}m ones (see \cite[Section 4]{HammHuangPer}).

\subsection{Layout}

The rest of the paper proceeds as follows: Section \ref{SEC:Background} contains the relevant notation and a characterization theorem for CUR decompositions which motivates some of the algorithmic aspects of the sequel; Section \ref{SEC:ColumnSelection} discusses deterministic and random column and row selection methods, and contains the precise statement of Theorem \ref{THM:IntroB}.  Section \ref{SEC:Stability} contains precise statements of our stability results in Theorem \ref{THM:IntroC} and Corollary \ref{COR:IntroA}, and Section \ref{SEC:LevStability} contains the precise statement of Theorem \ref{THM:IntroLev}.  The tie-in to Subspace Clustering is in Section \ref{SEC:SC}, and a summary of the different results and complexities along with a discussion of benefits and drawbacks of each is contained in Section \ref{SEC:Complexities}, while the remainder of the sections contain the proofs of the main results.

\section{Background}\label{SEC:Background}

\subsection{Notations}\label{SEC:Notation}

We will use $\K$ to represent either the real or complex field.  Any matrix $A\in\K^{m\times n}$ has a Singular Value Decomposition of the form $A=W\Sigma V^*$, where $W$ and $V$ are orthogonal matrices, and $\Sigma$ has entries only along its diagonal (the eigenvalues of $A^*A$ or equivalently of $AA^*$) which are the \textit{singular values} of $A$, and which are assumed to be in decreasing order and are denoted by $\sigma_{\max}=\sigma_1\geq\sigma_2\geq\dots \geq\sigma_{\min}=\sigma_{\rank(A)}>0$ (the rest of the singular values, if any, are $0$, but we will be concerned primarily with rectangular matrices and so will consider $\sigma_{\min}$ to be the minimal nonzero singular value of $A$).  If the underlying matrix must be specified, we write $\sigma_i(A)$.  For a given matrix, $\kappa(A)$ will denote its generalized spectral condition number (e.g., \cite{demko1986condition}), i.e., \[\kappa(A):=\frac{\sigma_{\max}(A)}{\sigma_{\min}(A)} = \|A\|_2\|A^\dagger\|_2.\]

The symbol $[n]$ denotes the set $\{1,\dots,n\}$ for $n\in\N$.  Given $I\subset[m]$, $A(I,:)$ represents the row submatrix of $A$ according to the set $I$ (i.e., $A(I,:)\in\K^{|I|\times n})$.  The column submatrix $A(:,J)\in\K^{m\times|J|}$ is defined similarly for $J\subset[n]$, and $A(I,J)$ is the overlap of these two.

We use $a\gtrsim b$ to mean that $a\geq cb$ for some universal constant $c>0$.

\subsection{Characterization of CUR decompositions}

For the reader's convenience, we recall the following characterization of CUR decompositions of low-rank matrices given in \cite{HammHuangACHA}.

\begin{theorem}[\cite{HammHuangACHA}]\label{THM:CUR}
Let $A\in\K^{m\times n}$ be fixed, and let $I\subset[m]$, $J\subset[n]$.  Let $C=A(:,J)$ and $R=A(I,:)$ be column and row submatrices of $A$, respectively, and let $U=A(I,J)$ be their intersection.  Then the following are equivalent:
\begin{enumerate}[(i)]
    \item\label{COND:ranku} $\rank(U)=\rank(A)$
    \item\label{COND:CUR} $A=CU^\dagger R$
    \item\label{COND:CCARR} $A = CC^\dagger AR^\dagger R$
    \item\label{COND:Adagger} $A^\dagger = R^\dagger UC^\dagger$
    \item\label{COND:Ranks} $\rank(C)=\rank(R)=\rank(A)$.
\end{enumerate}
Moreover, if any of the equivalent conditions above hold, then $U^\dagger = C^\dagger AR^\dagger$.
\end{theorem}

An important note for the sequel is that Theorem \ref{THM:CUR} holds even when $I$ and $J$ are allowed to be subsets of indices with repetitions allowed, and thus, e.g., $C$ may contain repeated columns of $A$.  Additionally, the equivalence \eqref{COND:ranku}$\Leftrightarrow$\eqref{COND:Ranks} allows for algorithms which choose columns and rows in parallel rather than sequentially which still allow one to show that an exact CUR decomposition of the form \eqref{COND:CUR} or equivalently \eqref{COND:CCARR} is obtained.

\section{Column and Row Sampling for CUR Decompositions}\label{SEC:ColumnSelection}

Here, we tackle the problem of determining how to select columns and rows such that an exact CUR decomposition of a pure low-rank matrix $A$ may be obtained.  The methods considered here break down into two categories: deterministic and random sampling.  Typically randomized methods require one to oversample columns and rows to ensure an exact decomposition and so might naturally not always preferred; however, their complexity may be less than the deterministic algorithms and so may be more suitable for truly large-scale matrices.  We proceed by highlighting several procedures in each category, and end the section by comparing their overall complexities.

\subsection{Deterministic Sampling}
There are several deterministic methods for column subset selection; for example, the QR decomposition based algorithm of Voronin and Martinsson \cite{VoroninMartinsson} and the Discrete Empirical Interpolation Method (DEIM) of Sorensen and Embree \cite{chaturantabut2010nonlinear,SorensenDEIMCUR}.
In this section, we will present DEIM for choosing column and row submatrices of $A$ which guarantees an exact CUR decomposition for a low-rank matrix.

The DEIM algorithm chooses $k$ columns from $A\in\mathbb{K}^{m\times n}$ by viewing the columns of 
$V_k=[v_1\quad v_2\quad\ldots\quad v_k]$ one at a time, where $v_i$  is the right singular vector of $A$ corresponding to the $i$-th largest singular value of $A$.  The algorithm starts from the leading singular vector $v_1$, and the first index $p_1$ corresponds to the largest magnitude entry in $v_1$, i.e.,$|v_1(p_1)|=\|v_1\|_{\infty}. $ With $I_n$ being the $n\times n$ identity, set ${\bf p_1}=[p_1]$, $P_1=I_n(:,{\bf p_1})$, $V_1=[v_1 ]$, and define 
the projection operator $\mathcal{P}_1=v_1(P_1^Tv_1)^{-1}P_1^T$. 

Suppose we have $j-1$ indices, with 
\[{\bf p_{j-1}}=\begin{bmatrix}p_1\\ \vdots\\
p_{j-1}
\end{bmatrix},\quad P_{j-1}=I_n(:,{\bf p_{j-1}}),\quad V_{j-1}=[v_1 \quad \ldots\quad v_{j-1}],\]  and \[\mathcal{P}_{j-1}=V_{j-1}(P_{j-1}^TV_{j-1})^{-1}P_{j-1}^T.\] Define the residual $r_j=v_j-\mathcal{P}_{j-1}v_j$, and the next index $p_j$ is chosen such that $|r_j(p_j)|=\|r_j\|_\infty.$

After $\ell$ iterations in the DEIM algorithm, we have indices ${\bf p_\ell}$ which satisfy the property $\|V_k({\bf p_\ell},:)\|<\sqrt{\frac{m\ell}{3}}2^\ell$.

	
	
	

	

	
	
	
	

	

\begin{proposition}\label{Prop: DEIM_NF}
Let $A\in\K^{m\times n}$ have rank $k$, and let $J\subset[n]$ be the set of $k$ column indices given by implementing the DEIM algorithm. 
Set $C=A(:,J)$.  Then $\rank(C)=\rank(A)$.  Consequently, if $I\subset[m]$ is a set of row indices given by running the DEIM algorithm 
on $A^*$ and $U=A(I,J)$, $R=A(I,:)$, then $A=CU^\dagger R$.
\end{proposition}

\begin{proof}
By \cite[Lemma 3.2]{SorensenDEIMCUR}, we have that $\text{rank}(R)=\text{rank}(A)$ and $\text{rank}(C)=\rank(A)$. Therefore, $A=CU^\dagger R$ by Theorem \ref{THM:CUR}.
\end{proof}

\subsection{Randomized Sampling}\label{SEC:RandomSampling}


Here we ask the question: given a low rank matrix $A$, how should we choose rows and columns to obtain a valid CUR decomposition as in Theorem \ref{THM:CUR}? Due to the equivalence of \eqref{COND:CUR} and \eqref{COND:Ranks} in Theorem \ref{THM:CUR}, it suffices to choose columns and rows of $A$ independently such that the related matrices $C$ and $R$ have the same rank as $A$ itself. 

When randomly sampling columns, one might either do so via a Bernoulli random trial at each column, or alternatively might sample columns with or without replacement according to a given probability distribution over the indices.  Here, we focus on the latter setting and show that in most cases, mildly oversampling columns yields a valid CUR decomposition with high probability.  In particular, while one evidently must sample at least $k$ columns, we show that often sampling $O(k\log k)$ columns is effective.  Three primary sampling distributions are considered based on previous work -- uniform \cite{DemanetWu}, column lengths \cite{DKMIII,KannanVempala}, and leverage scores \cite{DMM08}.  They are defined thusly:
\[p_j^\text{unif}:=\frac{1}{n},\quad p_j^\text{col}:=\frac{\|A(:,j)\|_2^2}{\|A\|_F^2},\quad p_j^\text{lev,k}:= \frac{1}{k}\|V_k(j,:)\|_2^2, \quad j\in[n].\]
The distributions for the rows are defined analogously, with $V_k$ being replaced by $W_k$ (the left singular vectors) in the case of leverage scores; for notational purposes we denote these $q_i^\text{unif}$, $q_i^\text{row}$, and $q_i^{\text{lev},k}$ for $i\in[m]$.  Note that for leverage scores, $k$ does not have to be the rank of $A$ in general, but the parameter $k$ determines how much the right singular vectors are truncated.

Uniform sampling is the easiest and cheapest to implement, but it can fail to provide good results, especially given a sparse input matrix, for example.  On the other hand, leverage scores typically achieve the best performance because they capture the eigenspace structure of the matrix, but this comes at the cost of a higher computational load to compute the distribution as it requires computing the truncated SVD of the initial matrix.  Column/row length sampling typically lies between both of the others in terms of performance as well as computational complexity.


In the algorithms which sample in this manner, the number of rows and columns chosen is fixed and deterministic, but when Bernoulli trials are used one only knows the expected number which will be selected.

Here we must introduce the concept of the \textit{stable rank} \cite{tropp2009column}, also called \textit{numerical rank} \cite{Rudelson_2007}, of $A$, defined by \[\text{st.rank}(A):=\frac{\|A\|_F^2}{\|A\|_2^2} = \sum_{i=1}^{\rank(A)}\frac{\sigma_i(A)^2}{\sigma_1(A)^2}.\]  Evidently, $\text{st.rank}(A)\leq\rank(A)$.   

One of the primary reasons for considering the stable rank of a matrix is that it is stable under small perturbations (whereas the rank is certainly not).  In particular, if $\widetilde{A} = A+E$, then applications of the triangle inequality produce
\[ \text{st.rank}(A)\left(\dfrac{1-\|E\|_F/\|A\|_F}{1+\|E\|_2/\|A\|_F}\right)^2 \leq \text{st.rank}(\widetilde{A})\leq \text{st.rank}(A)\left(\dfrac{1+\|E\|_F/\|A\|_F}{1-\|E\|_2/\|A\|_2}\right)^2.\]  We see that if $\|E\|_F/\|A\|_F$ and $\|E\|_2/\|A\|_2$ are small, then the stable ranks of $\widetilde{A}$ and $A$ are close, implying the claim.

The following theorem shows that one may sample essentially $r\log r$ columns and rows of a matrix (with $r=\text{st.rank}(A)$) to obtain an exact CUR decomposition with high probability.

\begin{theorem}\label{THM:ColRowChnM}Let $A\in\mathbb{K}^{m\times n}$ have rank $k$ and stable rank $r$.  Let $\delta\in(0,1)$, and let $0<\eps<\kappa(A)^{-1}$. Let $d_1\in[m]$, $d_2\in[n]$ satisfy
	\[
		d_1,d_2\gtrsim \left(\frac{r}{\eps^4\delta}\right)\log\left(\frac{r}{\eps^4\delta}\right).
		\]
Choose $I\subset[m]$ by sampling $d_1$ rows of $A$ independently with replacement according to probabilities $q_i^\textnormal{row}$ and choose $J\subset[n]$ by sampling $d_2$ columns of $A$ independently with replacement according to  $p_i^\textnormal{col}$.  Set $R=A(I,:)$, $C=A(:,J)$, and $U=A(I,J)$. 
Then with probability at least $(1-2\exp(-c/\delta))^2$,
\[ \rank(U)=k \text{ and } A=CU^\dagger R. \]
Moreover, the conclusion of the theorem also holds if we take $I_0$ and $J_0$ to be the indices of $I$ and $J$ above without repeated entries.
\end{theorem}

The proof of Theorem \ref{THM:ColRowChnM} is provided in Section \ref{SEC:ProofSampling}.

It should be noted that Theorem \ref{THM:ColRowChnM} utilizes the equivalent condition \eqref{COND:Ranks} of Theorem \ref{THM:CUR} which allows for a somewhat faster algorithm to sample columns and rows. Indeed, since we may check the ranks of $C$ and $R$ separately, we are able to choose columns and rows independently of each other and still guarantee an exact factorization.  One could lessen the sampling complexity of columns by first choosing rows and then choosing columns according to the stable rank of the row matrix $R$; however, this would require sequential sampling and would take more time.

Also note that the sampling complexity in Theorem \ref{THM:ColRowChnM} ostensibly depends on the stable rank of $A$, which is a reduction from most sampling methods for CUR approximations.  However, our assumption on $\eps$ implies that $\frac{r}{\eps^4}\geq k$, in which case our sampling complexity is at least
$\frac{k}{\delta}\log\left(\frac{k}{\delta}\right)$. 
Most results in the CUR approximation literature involving our choice of $U^\dagger$ (e.g., \cite{DKMIII}) require sampling $\frac{k}{\eps^\alpha\delta}\log(\frac{k}{\eps^\alpha\delta})$ rows and columns for some $\alpha$ where the $\eps$ is the same as ours.  Thus our sampling bound could not be derived from the existing ones without being of higher order.  There are more complicated choices for the matrix $U$ in CUR which yield lower sampling complexity; for example Boutsidis and Woodruff \cite{BoutsidisOptimalCUR} provide guarantees for sampling $O(k/\eps)$ columns and rows, but their CUR approximation is much more complicated than ours.  Remark \ref{REM:SamplingOrders} contains further discussion of our sampling orders.

\section{Stability of Column/Row Sampling}\label{SEC:Stability}

\subsection{Randomized Sampling}
While Theorem \ref{THM:ColRowChnM} is somewhat readily obtained from a previous analysis of Rudelson and Vershynin \cite{Rudelson_2007}, we extend their proof to illustrate that sampling with replacement is stable under perturbations of the probabilities, and moreover our analysis gives quantitative measures of said stability.  This brings us to our main stability theorem about exact CUR decompositions whose proof may be found in Section \ref{SEC:ProofStability}.
\begin{theorem}\label{THM:StableCUR}
Let $A\in\K^{m\times n}$ be fixed and have stable rank $r$ and rank $k$.  Suppose that $\widetilde{p},\widetilde{q}$ are probability distributions satisfying $\widetilde{p}_j\geq\alpha_j^2p_j^\textnormal{col}$ and $\widetilde{q_i}\geq\beta_i^2q_i^\textnormal{row}$ for all $i\in[m]$ and $j\in[n]$ for some constants $\alpha_i,\beta_i>0$ (with the convention that $\alpha_i=1$ if $A(:,j)=0$ and $\beta_i=1$ if $A(i,:)=0$).  Let $\alpha:=\min\alpha_i$, $\beta:=\min\beta_i$, and let $\gamma:=\min\{\alpha,\beta\}$.  Let $\delta\in(0,1)$ be given, and let $0<\eps<\min\{\kappa(A)^{-1},\delta^{-\frac14}\sqrt{2\gamma}\}$.
Let $d_1\in[m],d_2\in[n]$ satisfy 
\[ d_1,d_2\gtrsim \left(\frac{r}{\eps^4\delta}\right)\log\left(\frac{r}{\eps^4\delta}\right). \]
Choose $I\subset[m]$ by sampling $d_1$ rows of $A$ independently with replacement according to probabilities $\widetilde{q}_i$ and choose $J\subset[n]$ by sampling $d_2$ columns of $A$ independently with replacement according to $\widetilde{p}_i$.  Set $R=A(I,:)$, $C=A(:,J)$, and $U=A(I,J)$.  Then with probability at least $(1-2\exp(-\frac{c\alpha^2}{\delta}))(1-2\exp(-\frac{c\beta^2}{\delta}))$, the following hold:
\[ \rank(U) = k\;\;\text{and}\;\;A=CU^\dagger R.\]
Moreover, the conclusion of the theorem also holds if we take $I_0$ and $J_0$ to be the indices of $I$ and $J$ above without repeated entries.
\end{theorem}


\begin{remark}\label{REM:SamplingOrders}
Note that the assumptions on $\eps$ and $\delta$ imply certain relations, e.g.,
\[ \frac{r}{\eps^4\delta}\log\left(\frac{r}{\eps^4\delta}\right) = \Omega\left(\frac{k\kappa(A)^2}{\delta}\log\left(\frac{k\kappa(A)}{\delta}\right)\right), \]
and if we additionally assume that $\eps=\Theta(\kappa(A)^{-1})$, then
\[ \frac{r}{\eps^4\delta}\log\left(\frac{r}{\eps^4\delta}\right) = O\left(\frac{k\kappa(A)^4}{\delta}\log\left(\frac{k\kappa(A)}{\delta}\right)\right). \]


\end{remark}

\begin{remark}Note that this sampling complexity of essentially $k\log k$ for both columns and rows is better than previous results which required first row sampling of order $k^2\log k$ and then column sampling of order $|R|^2\log|R|$ with $|R|$ being the number of rows selected (e.g., \cite{DKMIII}).  The observation of Theorem \ref{THM:CUR}\eqref{COND:Ranks} allows sampling to be done independently and thus achieve lower complexity.\end{remark}

\subsection{Corollaries}

Theorem \ref{THM:StableCUR} admits many extensions.  First, we illustrate its conclusion for uniform and leverage score sampling.  The following essentially states that uniformly sampling rows and columns of $A$ still yields $A=CU^\dagger R$ with high probability in a certain sense; this is the first result of this kind that does not use any additional assumptions about the matrix $A$ such as coherency (e.g. \cite{DemanetWu}).

\begin{corollary}\label{COR:Uniform}
Let $A\in\K^{m\times n}$ have stable rank $r$ and rank $k$.  Let $\alpha:= \frac{1}{\sqrt{m}}\min\{\frac{\|A\|_F}{\|A(i,:)\|_2}:A(i,:)\neq0\}$ and $\beta:=\frac{1}{\sqrt{n}}\min\{\frac{\|A\|_F}{\|A(:,j)\|_2}:A(:,j)\neq0\}$, with $\gamma:=\min\{\alpha,\beta\}$.  Let $\delta\in(0,1)$ be given, and let $0<\eps<\min\{\kappa(A)^{-1},\delta^{-\frac14}\sqrt{2\gamma}\}$.  
Then sampling $d_1,d_2\gtrsim \left(\frac{r}{\eps^4\delta}\right)\log\left(\frac{r}{\eps^4\delta}\right)$ columns and rows of $A$ uniformly with replacement yields $C,U,$ and $R$ such that with probability at least $(1-2\exp(-\frac{c\alpha^2}{\delta}))(1-2\exp(-\frac{c\beta^2}{\delta}))$, 
\[\rank(U) = k\;\;\text{and}\;\;A=CU^\dagger R.\]
\end{corollary}
\begin{proof}
With the definitions of $\alpha,\beta$, we have that $\widetilde{q_i} = q_i^\textnormal{unif}=\frac1m \geq \alpha_i^2q_i^\textnormal{col}$ where $\alpha_i^2 = \frac1m \frac{\|A\|_F^2}{\|A(i,:)\|_2^2}$.  The analogous statement holds for $\widetilde{p}_j=p_j^\textnormal{unif}=\frac1n$, whereby an appeal to Theorem \ref{THM:StableCUR} yields the desired conclusion.
\end{proof}

\begin{remark}\label{REM:BoundedBelow}
It should be noted that the statement of Corollary \ref{COR:Uniform} can be generalized somewhat.  We will not write the full statement, but suppose $\widetilde{p}_j\geq\eps_0^2>0$ and $\widetilde{q}_i\geq\eps_1^2>0$, then defining $\alpha:=\eps_1^2\min\{\frac{\|A\|_F}{\|A(i,:)\|_2}:A(i,:)\neq0\}$ and $\beta$ analogously using $\eps_0$ and the column norms.  With these definitions in hand, the rest of the conclusion of Corollary \ref{COR:Uniform} holds for these parameters $\alpha$ and $\beta$.
\end{remark}

\begin{remark}
Embedded in the parameters of Corollary \ref{COR:Uniform} is an indication of the tradeoff between the sparsity of rows and columns of $A$ and the required sampling order, which one would expect to have in order to guarantee success of uniform sampling.  Indeed, consider the extreme case when $A$ consists of a single nonzero entry.  In this case, $\alpha=\frac{1}{\sqrt{m}}$, and the requirement on $\eps$ is such that approximately $m$ rows need to be sampled to guarantee that the single meaningful one is selected, which is the correct sampling order given that the rows are chosen uniformly at random. 
\end{remark}

The following shows the exact decomposition in the case of leverage score sampling.
\begin{lemma}\label{LEM:ColLev}
For $A\in\K^{m\times n}$ having rank $k$ and stable rank $r$, 
\[ p_j^{\textnormal{lev},k}\geq \frac{r}{k}p_j^\textnormal{col},\quad j\in[n].\]
The same inequality holds for $i\in[m]$ when comparing $q_i^\textnormal{row}$ with $q_i^{\textnormal{lev},k}$.
\end{lemma}
\begin{proof}
Let $j\in[n]$ be fixed but arbitrary, and let $A=W\Sigma V^*$ be the full SVD of $A$ and $A=W_k\Sigma_kV_k^*$ be its compact SVD.  By unitary invariance, we have
\[
    \|A(:,j)\|_2^2  = \|W\Sigma V^*(:,j)\|_2^2 = \|\Sigma V^*(:,j)\|_2^2.
\]
By the block structure of $\Sigma$, the latter quantity is $\|\Sigma_kV_k^*(:,j)\|_2^2 = \sum_{i}|\sigma_i|^2|V_k^*(i,j)|^2$ which is at most $\sigma_1^2\|V_k^*(:,j)\|_2^2$. The proof is complete upon dividing by $\|A\|_F^2$ and recalling that $p_j^{\textnormal{lev},k} = \frac1k\|V_k^*(:,j)\|_2^2$ and the definition of stable rank.  The statement and proof for the row leverage scores are identical.
\end{proof}

\begin{corollary}\label{COR:Leverage}
Let $A\in\K^{m\times n}$ have rank $k$ and stable rank $r$. Set $\alpha^2=\beta^2=\frac{r}{k}.$  Then with the notations and assumptions of Theorem \ref{THM:StableCUR}, sampling $d_1,d_2\gtrsim \left(\frac{r}{\eps^4\delta}\right)\log\left(\frac{r}{\eps^4\delta}\right)$ columns and rows of $A$ independently with replacement according to leverage scores, $p_j^{\textnormal{lev},k}$ and $q_i^{\textnormal{lev},k}$ yields 
\[\rank(U) = k\;\;\text{and}\;\;A=CU^\dagger R\]
with probability at least $(1-2\exp(-\frac{cr}{k\delta}))^2$.
\end{corollary}
\begin{proof}
Apply the conclusion of Lemma \ref{LEM:ColLev} in the statement of Theorem \ref{THM:StableCUR}.
\end{proof}

\subsection{Random Sampling of Noisy Matrices}

We now turn to analyzing what happens when we observe a noisy version of a low-rank matrix; i.e., we see $\widetilde{A}=A+E$ where $A$ is low rank, but $E$ is some noise matrix.  The primary question we ask is: if we select columns and rows $\widetilde{C}=\widetilde{A}(:,J)$ and $\widetilde{R}=\widetilde{A}(I,:)$ via a probability distribution determined by $\widetilde{A}$, can we succeed in obtaining a valid CUR decomposition of the underlying low-rank part, $A$.  That is, if $C=A(:,J)$ and $R=A(I,:)$, do we have $A=CU^\dagger R$?

First, regard that Corollary \ref{COR:Uniform} implies an affirmative answer to this question as uniform sampling doesn't see the difference between $\widetilde{A}$ and $A$.  However, we may obtain another kind of stability from Theorem \ref{THM:StableCUR} in this vein.

Suppose that $\widetilde{p}_j^\textnormal{col} = \frac{\|\widetilde{A}(:,j)\|_2^2}{\|\widetilde{A}\|_F^2},$ and similarly for $\widetilde{q}_i^\textnormal{row}$.  Then via the same calculation as was done for the stable rank of $\widetilde{A}$ compared with that of $A$, we have 
\begin{equation}\label{EQN:qtilde} \widetilde{q}_i^\textnormal{row} = \frac{\|\widetilde{A}(i,:)\|_2^2}{\|\widetilde{A}\|_F^2} \geq \left(\frac{1-\frac{\|E(i,:)\|_2}{\|A(i,:)\|_2}}{1+\frac{\|E\|_F}{\|A\|_F}}\right)^2\frac{\|A(i,:)\|_2^2}{\|A\|_F^2}=:\beta_i^2q_i^\textnormal{row},\end{equation}
\begin{equation}\label{EQN:ptilde} \widetilde{p}_j^\textnormal{col} \geq \left(\frac{1-\frac{\|E(:,j)\|_2}{\|A(:,j)\|_2}}{1+\frac{\|E\|_F}{\|A\|_F}}\right)^2p_i^\textnormal{col} =: \alpha_j^2p_j^\textnormal{col}. \end{equation}
In the event that a row or column of $A$ is identically 0, then we set $\alpha_j$ or $\beta_i$ to be 1 by convention.  Note that to apply Theorem \ref{THM:StableCUR}, we must require that no column or row of $\widetilde{A}$ can be 0 when the corresponding column or row of $A$ is nonzero.  Otherwise, there would be some $i$ or $j$ for which $\beta_i=0$ or $\alpha_j=0$, whereby the result would not apply.

Now we can conclude that sampling columns and rows of $\widetilde{A}$ according to column and row lengths can ensure that the underlying CUR decomposition is valid for $A$ as long as $\eps,\delta$ are small enough and the obstacle mentioned above is not present. 

\begin{corollary}\label{COR:NoiseStability}
Let $\widetilde{A}=A+E$ with $A$ having rank $k$ and stable rank $r$, and suppose that no row or column of $\widetilde{A}$ is 0 unless the corresponding row or column of $A$ is 0.  Let $\widetilde{q}_i^\textnormal{row}, \widetilde{p}_j^\textnormal{col},\alpha_j,\beta_i$ be as in \eqref{EQN:qtilde}, \eqref{EQN:ptilde}, and set $\alpha=\min \alpha_j$ and $\beta=\min\beta_i$.  Then with the notations and assumptions of Theorem \ref{THM:StableCUR}, sampling columns of $\widetilde{A}$ independently with replacement according to the given probabilities yields $I$ and $J$ such that if $C=A(:,J), U=A(I,J),$ and $R=A(I,:)$, then
\[\rank(U) = k\;\;\text{and}\;\;A=CU^\dagger R\]
with high probability.
\end{corollary}

As a practical note on using this corollary, one could try to estimate the signal-to-noise ratio to obtain an estimate for the values $\alpha_j$ and $\beta_i$ in the above expressions.  This would then give an indication of how to choose $\eps,\delta$.  

\begin{remark}One application of Corollary \ref{COR:NoiseStability} is to perturbation bounds given in \cite{HammHuangPer}.  Therein, it was shown that $\|A-\widetilde{C}\widetilde{U}^\dagger\widetilde{R}\|\leq \|A-CU^\dagger R\|+O(\|E\|)$ where $\|\cdot\|$ was any Schatten $p$--norm.  Corollary \ref{COR:NoiseStability} thus implies that with high probability, sampling noisy columns and rows of $\widetilde{A}$ according to their lengths yields $\|A-\widetilde{C}\widetilde{U}^\dagger\widetilde{R}\| = O(\|E\|)$ with high probability (see \cite{HammHuangPer} for more detailed estimates of the error in terms of $\|E\|$).
\end{remark}

\subsection{Deterministic Sampling of Noisy Matrices}
Proposition \ref{Prop: DEIM_NF} shows that the DEIM algorithm recovers a low-rank matrix exactly. Here, we will analyze the stability of DEIM in the vein of the previous subsection and give a quantitative relationship between the underlying low-rank matrix and the noise.

\begin{proposition}
Let $\widetilde{A}=A+E$ with $A$ having rank $k$. Let $A=W_k\Sigma_k V_k^*$ and denote by $\widetilde{A}_k=\widetilde W_k\widetilde \Sigma_k\widetilde V_k^*$ the truncated SVD of $\widetilde{A}$.  Suppose that $I\subset[m],J\subset[n]$ are chosen by DEIM algorithm 
on $\widetilde W_k$ and $\widetilde V_k$ respectively, and set $C=A(:,J)$, $R=A(I,:)$, and $U=A(I,J)$. If $\sigma_k(A)\geq\left(1+2^k\sqrt{\frac{\max\{nk,mk\}}{3}}\right)\|E\|_2 $, then $A=CU^\dagger R$.
\begin{proof}
Let $\widetilde{A}=\widetilde{W}\widetilde{\Sigma}\widetilde{V}^*$ with $\widetilde{W}=\begin{bmatrix}\widetilde{W}_k&\widetilde{W}_{\perp}\end{bmatrix}$, $\widetilde{\Sigma}=\begin{bmatrix} 
\widetilde{\Sigma}_k&0\\
0&\widetilde{\Sigma}_{\perp}
\end{bmatrix}$  and $\widetilde{V}=\begin{bmatrix}\widetilde{V}_k&\widetilde{V}_{\perp}\end{bmatrix}$.  Note that by an inequality due to Weyl (see, e.g., \cite[Corollary 8.6.2]{GolubVanLoan}), $\sigma_k(R)+\|E(I,:)\|_2\geq
\sigma_k(\widetilde{R})=\sigma_k(\widetilde{A}(I,:))$, which is
\begin{eqnarray*}
\sigma_k(\widetilde{W}(I,:)\widetilde{\Sigma}\widetilde{V}^*)&=&\sigma_k(\widetilde{W}(I,:)\widetilde{\Sigma})\\
&=&\sigma_k\left(\begin{bmatrix}\widetilde{W}_k(I,:)\widetilde{\Sigma}_k& \widetilde{W}_{\perp}(I,:)\widetilde{\Sigma}_{\perp}
\end{bmatrix}\right)\\
&\geq&\sigma_k(\widetilde{W}_k(I,:)\widetilde{\Sigma}_k)\\
&\geq& \sigma_k(\widetilde{W}_k(I,:))\sigma_k(\widetilde{A}).
\end{eqnarray*}
Thus, by \cite[Lemma 4.4]{SorensenDEIMCUR}, we have $\sigma_k(R)\geq\sigma_k(\widetilde{W}_k(I,:))\sigma_k(\widetilde{A})-\|E(I,:)\|_2\geq\sqrt{\frac{3}{mk}}\frac{1}{2^k}\sigma_k(\widetilde{A})-\|E(I,:)\|_2$. Similar, we can prove that $\sigma_k(C)\geq\sqrt{\frac{3}{nk}}\frac{1}{2^k}\sigma_k(\widetilde{A})-\|E(:,J)\|_2$. Since  $\sigma_k(A)\geq(1+2^k\sqrt{\frac{\max\{nk,mk\}}{3}})\|E\|_2 $, we have $\sigma_k(R)>0$ and $\sigma_k(C)> 0$ (by utilizing the fact that $\sigma_k(\widetilde{A})\geq\sigma_k(A)-\|E\|_2$. Hence, $\rank(C)=\rank(R)=\rank(A)$, which implies $A=CU^\dagger R$ by Theorem \ref{THM:CUR}.
\end{proof}
\end{proposition}

\section{Stability Based on Leverage Scores}\label{SEC:LevStability}

The previous sections gave a notion of stability of column and row sampling methods in which stability was with respect to the column and row norms.  Here, we show that stability in terms of leverage scores may also be obtained.  We note that these notions of stability appear to not follow from one another, but rather from quite different techniques, and so we present both notions here and leave it to the reader to choose the appropriate result for their purposes.

A strong notion of stability comes from the following theorem of Yang et.~al.~\cite{Yang}.  To state their result, we need the following quantity:
\[ c(\widetilde{p}):=\underset{j\in[n]}\max\;\frac{p_j^{\textnormal{lev},k}}{\widetilde{p}_j}.\]
The following is a special case of \cite[Theorem 1]{Yang}.
\begin{theorem}\label{THM:GeneralSamplingInequality}
Let $A\in\K^{m\times n}$ have rank $k$, $\delta\in(0,1)$, and suppose $C\in\K^{m\times \ell}$ is a column submatrix whose $\ell$ columns are sampled from $[n]$ independently with replacement according to probabilities $\widetilde{p}$.  Then with probability $1-e^{-\frac1\delta}-2k\exp(-\frac{\ell}{8kc(\widetilde{p})})$, $A=CC^\dagger A$, hence $\rank(C)=k$. 
\end{theorem}

Of course one may obtain a guarantee that $A=CU^\dagger R$ via the obvious method as was done above.

Now we may deduce stability of CUR decompositions from leverage score sampling as follows.

\begin{corollary}\label{COR:LeverageStability}
Suppose $A\in\K^{m\times n}$ has rank $k$. Let $\delta\in(0,1)$, and suppose that there exists $\beta_1,\beta_2\in(0,1]$ such that $\widetilde{p}_j\geq\beta_1 p_j^{\textnormal{lev},k}$ for all $j\in[n]$ and $\widetilde{q}_i\geq\beta_2q_i^{\textnormal{lev},k}$ for all $i\in[m]$.  Let 
\[ \ell_t = \frac{8}{\beta_i}\left(\log(2k)+\frac1\delta\right)k,\quad t=1,2.\]
If $C\in\K^{m\times\ell_1}$ is a column submatrix of $A$ whose columns are sampled from $[n]$ independently with replacement according to probabilities $\widetilde{p}$, and $R\in\K^{\ell_2\times n}$ is a row submatrix of $A$ whose rows are sampled from $[m]$ independently with replacement according to $\widetilde{q}$, then with probability at least $(1-2e^{-\frac1\delta})^2$, $A=CU^\dagger R$.
\end{corollary}

\begin{proof}
The assumption on the probability distributions ensures that $c(\widetilde{p})\leq\beta_1^{-1}$, and similarly for the row sampling distribution.  Thus the probability of success of column sampling is at least $1-e^{-\frac1\delta}-2k\exp\left(-\frac{\ell\beta_1}{8kc(\widetilde{p})}\right),$ and a similar change of success for row sampling.  The choice of $\ell_1,\ell_2$ ensures that these success probabilities are each at least $1-2e^{-\frac1\delta}$, which implies the result.
\end{proof}

Note that one can easily state a result similar to Corollary \ref{COR:LeverageStability} under the assumption that $\widetilde{p}_j,\widetilde{q}_i\geq\beta>0$ as was mentioned in Remark \ref{REM:BoundedBelow}. The terms $\beta_1,\beta_2$ will then be related to the Leverage Scores of $A$.  Additionally, the analysis of \cite{Yang} is stronger than that used for proving stability in terms of column and row lengths, and admits a better sampling order for the latter probabilities in some cases.  Indeed, we have the following.

\begin{lemma}\label{LEM:ColLev2}
Suppose $A\in\K^{m\times n}$ has rank $k$ and stable rank $r$. Then
\[ p_j^{\textnormal{col}}\geq \frac{k}{r\kappa(A)^2}p_j^{\textnormal{lev},k},\quad j\in[n], \]
and the same inequality holds for $i\in[m]$ relating $q_i^{\textnormal{row}}$ to $q_i^{\textnormal{lev},k}$.
\end{lemma}

\begin{proof}
Similar to the proof of Lemma \ref{LEM:ColLev}, we have that $\|A(:,j)\|_2^2 = \sum_i|\sigma_i|^2|V_k^*(i,j)|^2\geq\sigma_k^2\|V_k^*(:,j)\|_2^2$.  This implies that 
\[ p_j^{\textnormal{col}}\geq k\frac{\sigma_k^2\sigma_1^2}{\sigma_1^2\|A\|_F^2}p_j^{\textnormal{lev},k} = \frac{k}{r\kappa(A)^2}p_j^{\textnormal{lev},k}. \]
\end{proof}

\begin{corollary}\label{COR:ColLev}
Suppose $A\in\K^{m\times n}$ has rank $k$ and stable rank $r$, and let $\delta\in(0,1)$.  Then sampling $8r\kappa(A)^2(\log(2k)+\frac1\delta)$ columns and rows independently with replacement according to $p_j^{\textnormal{col}}$ and $q_i^{\textnormal{row}}$, respectively, yields $A=CU^\dagger R$ with probability at least $(1-2e^{-\frac1\delta})^2$.
\end{corollary}

\begin{proof}
Use the conclusion of Lemma \ref{LEM:ColLev2} in Corollary \ref{COR:LeverageStability}.
\end{proof}

Note that the sampling order given by Corollary \ref{COR:ColLev} is typically better than that of Theorem \ref{THM:ColRowChnM} (cf. Remark \ref{REM:SamplingOrders}).

\section{Sampling Guarantees for Subspace Clustering}\label{SEC:SC}

Here we give an application of the sampling methods described in the previous subsections to the \textit{Subspace Clustering Problem} \cite{vidal2011subspace}. In many applications, including motion segmentation \cite{CosteiraKanade}, facial recognition \cite{basri2003lambertian}, and cryo-electron microscopy \cite{hadani2011representation}, data is well-modeled to lie on or near a union of low-dimensional subspaces in the ambient space.  That is, given data $A\in\K^{m\times n}$ consists of columns that lie in $\mathscr{U}=\bigcup_{i=1}^\ell S_i$ where each $S_i$ is an affine subspace of $\K^m$.  The goal is to cluster data according to the subspaces, i.e., to find an assignment function
\[ \Pi:[n]\to[\ell],\quad \textnormal{such that} \quad \Pi(i) = j \textnormal{  iff  } a_i\in S_j. \]
Assuming enough data from each subspace is contained in $A$ to uniquely determine the subspace, the assignment function then allows one to obtain a basis for each subspace from the partition by, e.g., Principal Component Analysis.

Matrix factorization methods have been used to good effect in solving the Subspace Clustering Problem \cite{AHKS,ACKS,CosteiraKanade} as have associated low-rank based optimization methods \cite{SSC,liu2012robust}. In particular, it is known that under certain subspace configurations, the truncated SVD, any basis factorization with basis vectors coming from the subspaces $S_i$, and CUR decompositions can all be used to give a valid clustering of the data. For a longer discussion, the reader may consult \cite{AHKS}. Here we illustrate how random sampling may be used to guarantee a CUR-based solution to the subspace clustering problem.

To state our results, let us start with some definitions. First of all, without loss of generality, we may assume the subspaces are linear (given affine subspaces of $\K^{m}$, one may consider the linear subspaces of $\K^{m+1}$ spanned by the elements of $S_i$ in homogeneous coordinates).  A collection $B$ of points in a subspace of dimension $d$ are called \textit{generic} provided any collection of $d$ points from $B$ are linearly independent.  A collection of subspaces $\{S_1,\dots,S_\ell\}$ of $\K^n$ is said to be \textit{independent} provided $\dim(\sum_{i=1}^\ell S_i)=\sum_{i=1}^\ell \dim(S_i)\leq n$.  Given a collection of data $A$ from $\mathscr{U}=\bigcup_{i=1}^\ell S_i$, a matrix $W$ is called a \textit{clustering matrix} if $W_{i,j}\neq0$ if and only if $a_i$ and $a_j$ are in the same subspace.  Finally, by $|A|$, we mean the matrix whose entries are the absolute values of the entries of $A$.  The main theorem of \cite{AHKS} is the following, which states that any valid CUR decomposition of subspace data under certain assumptions gives rise to a clustering matrix for the data.

\begin{theorem}[{\cite[Theorem 2]{AHKS}}]\label{THM:SC}
Suppose that $A\in\K^{m\times n}$ has columns which come from a union of linear subspaces $\mathscr{U}=\bigcup_{i=1}^\ell S_i$ which are independent, and the data from each subspace is generic.  Let $d_{\max}=\max\dim(S_i)$.  Let $A=CU^\dagger R =: CY$ be any CUR decomposition of $A$, and let $Q:=|Y^*Y|$.  Then $Q^{d_{\max}}$ is a clustering matrix for $A$.  
\end{theorem}

\begin{corollary}\label{COR:SC}
Let $A$ be as in Theorem \ref{THM:SC}, and suppose that columns and rows of $A$ are sampled according to Theorem \ref{THM:ColRowChnM}.  Then with high probability, $A=CU^\dagger R$, and $Q^{d_{\max}}$ defined as in Theorem \ref{THM:SC} is a clustering matrix for $A$.
\end{corollary}

\section{Summary of Complexities}\label{SEC:Complexities}

Having discussed several deterministic and randomized column sampling schemes, it is pertinent to illustrate the advantages and drawbacks of each method. The following table shows sampling complexities (i.e., how many columns and rows must be sampled) as well as the overall algorithmic complexity which takes into account the cost of forming the probability distributions but does not account for forming the matrix $U^\dagger$ or for the multiplication $CU^\dagger R$ since this is the same over all of the results above. 

\begin{table}[h]
\begin{tabular}{|c|c|c|c|c|c|}
\hline
Sampling & \# Rows & \# Cols & Success Prob & Complexity & Ref \\ \hline
$p_i^\textnormal{unif}, q_i^\textnormal{unif}$ & $\frac{r}{\eps^4\delta}\log(\frac{r}{\eps^4\delta})$ & $\frac{r}{\eps^4\delta}\log(\frac{r}{\eps^4\delta})$ & $(1-2e^{-c\gamma^2/\delta})^2$ & $O(1)$ & Cor \ref{COR:Uniform}\\ \hline
$p_i^\textnormal{col}, q_i^\textnormal{col}$ & $\frac{r}{\eps^4\delta}\log(\frac{r}{\eps^4\delta})$ & $\frac{r}{\eps^4\delta}\log(\frac{r}{\eps^4\delta})$ & $(1-2e^{-c/\delta})^2$ & $O(mn)$ & Thm \ref{THM:ColRowChnM}\\ \hline
$p_i^\textnormal{col}, q_i^\textnormal{col}$ & $r\kappa(A)^2(\log(k)+\frac{1}{\delta})$ & $r\kappa(A)^2(\log(k)+\frac{1}{\delta})$ & $(1-2e^{-c/\delta})^2$ & $O(mn)$ & Cor \ref{COR:ColLev}\\ \hline

$p_i^{\textnormal{lev},k}, q_i^{\textnormal{lev},k}$ & $\frac{k^2}{\eps^2\delta}$ & $\frac{k^2}{\eps^2\delta}$ & $1-e^{-1/\delta}$ & $O(\textnormal{SVD}(A,k))$ & \cite{DMM08}\\ \hline
$p_i^{\textnormal{lev},k}, q_i^{\textnormal{lev},k}$ & $k\log(k)+\frac{k}{\delta}$ & $k\log(k)+\frac{k}{\delta}$ & $(1-2e^{-1/\delta})^2$ & $O(\textnormal{SVD}(A,k))$ & \cite{Yang}\\ \hline
DEIM-CUR & $k$ & $k$ &1&$O(\textnormal{SVD}(A,k) + k^4)$& \cite{SorensenDEIMCUR}\\
\hline
RRQR-CUR & $k$ & $k$ &1&$O(\textnormal{RRQR}(A))$& \cite{VoroninMartinsson} \\ \hline
\end{tabular}\caption{Table summarizing sampling complexities for different algorithms.}
\end{table}

Note that the asymptotic complexity for SVD$(A,k)$ and RRQR$(A)$ are both $O(mnk)$.  

\section{Proof of Theorem \ref{THM:ColRowChnM}}\label{SEC:ProofSampling}
	
To supply the proof, we first need some simple lemmas including the following that is derived from the proof of \cite[Theorem 1.1]{Rudelson_2007}. In what follows, by \textit{rescaled rows} of $A$, we mean that a row in $\hat{R}$ in subsequent results corresponds to a row of $\hat{A}$, whose $i$--th row is $\frac{\|A\|_F}{\sqrt{d}\|A(i,:)\|_2}A(i,:)$.  Rescaled columns of $\hat{A}$ are defined analogously.
	
\begin{proposition}[\cite{Rudelson_2007}]\label{THM_Rudelson}
		Let $A\in\K^{m\times n}$ have stable rank $r$.  Let $\eps,\delta\in(0,1)$, and let $d\in [m]$ satisfy
		\[
		d\gtrsim \left(\frac{r}{\eps^4\delta}\right)\log\left(\frac{r}{\eps^4\delta}\right).
		\]
		Consider a $d\times n$ matrix $\hat{R}$, which consists of $d$ 
		rescaled rows of $A$ 
		picked independently with replacement according to $q_i^{\textnormal{row}}.$
		Then with probability at least $1-2\exp(-c/\delta)$, 
		\[
		\|A^*A-\hat{R}^*\hat{R}\|_2\leq\frac{\eps^2}{2}\|A\|_2^2.
		\]
\end{proposition}

\begin{proof}[\textbf{Proof of Theorem \ref{THM:ColRowChnM}}]
Note that by Theorem \ref{THM:CUR} it suffices to show that $\rank(C)=\rank(R)=k$.  To utilize Proposition \ref{THM_Rudelson}, let $\hat{C}$ and $\hat{R}$ be rescaled versions of $C$ and $R$, respectively, and note that $\rank(\hat{R})=\rank(R)$ and $\rank(\hat{C})=\rank(C)$.  By Proposition \ref{THM_Rudelson} and the assumption on $\eps$, with probability at least $1-2\exp(-c/\delta)$ the following holds: 
		\begin{equation}\label{EQ:RudelsonVershynin}
		\|A^*A-\hat{R}^* \hat{R}\|_2\leq \frac{\eps^2}{2}\|A\|_2^2< \frac{1}{2}\sigma_{k}^2(A)<\sigma_k^2(A).
		\end{equation} 
Therefore $\rank(\hat{R})\geq k$.  In addition, $\rank(\hat{R})\leq\rank(A)=k$, hence equality holds.	

Using the same argument again, we can conclude that with probability at least $1-2\exp(-c/\delta)$,  $\rank(\hat{C})=k$.  Thus with probability at least $\left(1-2\exp(-c/\delta) \right)^2$, $\rank(C)=\rank(R)=k$, and so $A=CU^\dagger R$.   The moreover statement follows from the fact that repeated columns and rows do not affect the validity of the statement $A=CU^\dagger R$ as mentioned subsequent to Theorem \ref{THM:CUR}. 
\end{proof}

\section{Proof of Theorem \ref{THM:StableCUR}}\label{SEC:ProofStability}

To prove the main stability theorem for sampling, we first need the following modification of Proposition \ref{THM_Rudelson}.
\begin{theorem}\label{THM:Stability}
Let $A\in\K^{m\times n}$ be fixed and have stable rank $r$.  Suppose that $\widetilde{q}$ is a probability distribution satisfying $\widetilde{q_i}\geq\alpha_i^2 q_i^\text{row}$ for all $i\in[m]$ for some constants $\alpha_i>0$ (with the convention that $\alpha_i=1$ if $A(i,:)=0$).  Let $\alpha:=\min\alpha_i$, $\delta\in(0,1)$, and let  $0<\eps<\delta^{-\frac14}\sqrt{2\alpha}$.
Let $d\in[m]$ satisfy 
\[ d\gtrsim \left(\frac{r}{\eps^4\delta}\right)\log\left(\frac{r}{\eps^4\delta}\right), \]
and let $\hat{R}$ be a $d\times n$ matrix consisting of rescaled rows of $A$ chosen independently with replacement according to $\widetilde{q}$.  Then with probability at least $1-2\exp(-\frac{c\alpha^2}{\delta})$ 
\[\|A^*A-\hat{R}^*\hat{R}\|_2\leq\frac{\eps^2}{2}\|A\|_2^2. \]
\end{theorem}

The proof of Theorem \ref{THM:Stability} requires a simple modification of the proof of the main theorem in \cite{Rudelson_2007}.  For completeness, we give the proof here; the first ingredient is the following.

\begin{theorem}[{\cite[Theorem 3.1]{Rudelson_2007}}]\label{THM:RVLargeNumbers}
Let $y$ be a random vector in $\K^n$ which is uniformly bounded almost everywhere, i.e. $\|y\|_2\leq M$.  Assume for normalization that $\|\E(y\otimes y)\|_2\leq1$.  Let $y_1,\dots,y_d$ be independent copies of $y$.  Let \[ a:= C_0\sqrt{\frac{\log d}{d}}M. \] Then
\begin{enumerate}
    \item[(i)] If $a<1$, then \[ \E\left\|\frac1d\sum_{i=1}^d y_i\otimes y_i-\E(y\otimes y) \right\|_2\leq a;\]
    \item[(ii)] For every $t\in(0,1)$, \[ \Prob\left\{\left\|\frac1d\sum_{i=1}^d y_i\otimes y_i-\E(y\otimes y)\right\|_2>t\right\} \leq 2\exp(-ct^2/a^2). \]
\end{enumerate}
\end{theorem}

Note that Theorem \ref{THM:RVLargeNumbers} was proved in \cite{Rudelson_2007} for $\K=\R$, but the proof is valid without change for complex vectors, which we need for our application \cite{RudelsonPrivate}.

\begin{proof}[Proof of Theorem \ref{THM:Stability}]
Without loss of generality, suppose that $\|A\|_2=1$.  Let $x_i$ be the rows of $A$ so that $A^*A = \sum_{i=1}^m x_i\otimes x_i$.  Define the random vector $y$ via
\[ \Prob\left(y = \frac{1}{\sqrt{\widetilde{p}_i}}x_i\right) = \widetilde{p}_i.\]
Note that by assumption on $\widetilde{p}$, $\Prob(y=x_i) = 0$ only if $x_i=0$.  Let $y_1,\dots,y_d$ be independent copies of $y$, and let $\hat{A}$ be the matrix whose rows are $\frac{1}{\sqrt{d}}y_i$. Then we have $\hat{A}^*\hat{A} = \frac1d\sum_{i=1}^d y_i\otimes y_i$, and  $\E(y\otimes y) = A^*A$; indeed
\[\E(y\otimes y) = \sum_{i=1}^m \frac{1}{\sqrt{\widetilde{p}_i}}x_i\otimes \frac{1}{\sqrt{\widetilde{p}_i}}x_i\widetilde{p}_i = \sum_{i=1}^m x_i\otimes x_i = A^*A.  \]
Now by assumption on $\widetilde{p}$, we may choose \[\|y\|_2 = \frac{\|x_i\|_2}{\alpha_i\|x_i\|_2}\|A\|_F \leq \frac1\alpha\|A\|_F = \frac{\sqrt{r}}{\alpha}=:M.\]  Applying Theorem \ref{THM:RVLargeNumbers} with the assumption on $d$ yields (as in \cite{Rudelson_2007})
\[ a = \frac{1}{\alpha}C\left(\frac{\log d}{d}r\right)^\frac12 \leq \frac{\eps^2\sqrt{\delta}}{2\alpha}.\]
This quantity is thus bounded by 1 provided $\frac{\eps^2\sqrt{\delta}}{2}<\alpha$. In this event, Theorem \ref{THM:RVLargeNumbers} $(ii)$ implies that if $t=\frac{\eps^2}{2}$, then \[\|A^*A-\hat{A}^*\hat{A}\|_2\leq\frac{\eps^2}{2} \]
with probability at least $1-2\exp(-\frac{c\alpha^2}{\delta})$.  Note also that by the assumption on $\eps,\delta$, and $\alpha$, we have $\frac{\alpha^2}{\delta}>\frac{\eps^4}{4}$, whence the given event holds with probability at least $1-2\exp(-c\eps^4)$.
\end{proof}

\begin{proof}[Proof of Theorem \ref{THM:StableCUR}]
The proof is the same as that of Theorem \ref{THM:ColRowChnM} \textit{mutatis mudandis}, where one applies Theorem \ref{THM:Stability} rather than Theorem \ref{THM_Rudelson} to conclude that $C$ and $R$ have rank $k$.
\end{proof}

\section*{Acknowledgements} K. H. is partially supported by the National Science Foundation TRIPODS program, grant number NSF CCF--1740858. 
LX.H is partially supported by NSF CAREER DMS 1348721 and NSF BIGDATA 1740325.
\bibliographystyle{plain}
\bibliography{HammHuang}

\end{document}